\documentclass[french,11pt,a4paper]{amsart}

\usepackage{amsmath}
\usepackage{amsfonts}
\usepackage{amssymb}
\usepackage{amsthm}
\usepackage{a4wide}
\usepackage{mathrsfs}
\usepackage{xcolor}

\usepackage[utf8]{inputenc}  
\usepackage[T1]{fontenc}

\usepackage{graphicx}

\newtheorem{theorem}{Théorème}
\newtheorem{dfn}{Définition}
\newtheorem{lem}{Lemme}

\newtheorem{exm}{Exemple}

\newtheorem{prop}{Proposition}

\def\C{\mathbb{C}}

\def\N{\mathbb{N}}
\def\di{\displaystyle}
\def\K{\mathbb{K}}
\def\R{\mathbb{R}}

\usepackage{amsaddr}

\title[Symétries de Lie et tissus implicites du plan]{Application des groupes de Lie à la recherche des symétries des tissus implicites du plan}

\author{Jacky Cresson}
\address{Université de Pau et des Pays de l'Adour - E2S, Laboratoire de Mathématiques et de leurs Applications, UMR CNRS 5142, Batiment IPRA, avenue de l'Université, 64000 Pau, France}
\author{Jordy Palafox}
\address{CY Tech, Département de Mathématiques, 2 Bd Lucien, 64000 Pau, France}

\email{jacky.cresson@univ-pau.fr, jordy.palafox@cyu.fr}

\begin{document}

\maketitle



\setcounter{tocdepth}{3}

\begin{abstract}
On retrouve des résultats de Alain Hénaut sur les groupes de symétrie des tissus implicites dans le cadre usuel des actions de groupes de Lie sur les équations différentielles. On donne aussi un lien entre l'algèbre de Lie des symétries d'un tissu et l'existence de polynômes de Darboux.
\end{abstract}


\section{Introduction}

Dans \cite{henaut}, Alain Hénaut obtient via des outils de géométrie algébrique et/ou différentielle une {\bf caractérisation des tissus du plan pararallélisables} via la dimension de leur {\bf groupe de symétries}. Dans cet article, nous donnons des démonstrations alternatives de ses résultats dans le {\bf cadre classique des symétries d'équations différentielles} exposé notamment par P.J. Olver dans son livre \cite{olver1}. Par ailleurs, l'{\bf algèbre de Lie du groupe des symétries d'un tissu} est une sous-algèbre de Lie du {\bf module de dérivation} associé à la {\bf courbe discriminante du tissu}. Ce module possède une {\bf interprétation dynamique} : il correspond à l'ensemble des champs de vecteurs laissant invariant (au sens dynamique) la courbe. On montre que le {\bf discriminant de la courbe} est un {\bf polynôme de Darboux} pour ces champs. L'article est organisé de la manière suivante:\\

La Section \ref{tissusimplicites} donne la définition des tissus implicites et les principaux objets qui y sont attachés. Dans la Section \ref{14}, après un bref rappel sur les groupes de symétrie des équations différentielles suivant P-J. Olver \cite{olver1}, nous donnons une caractérisation explicite des générateurs infinitésimaux des groupes de symétries d'un tissu implicite en fonction de son polynôme de présentation. Dans la Section \ref{sectionexemple}, des exemples de calculs effectifs de ces groupes de symétries sont donnés pour des tissus particuliers (parallèle, de Clairaut, de Zariski, etc). Dans la Section \ref{19}, nous explicitons le lien entre les symétries et l'algèbre de dérivations de la courbe discriminante en utilisant le théorie classique de Darboux sur les courbes invariantes d'équations différentielles. Enfin la Section \ref{perspective} discute quelques perspectives de ce travail.

\section{Tissus implicites du plan}
\label{tissusimplicites}

La géométrie des tissus est l'étude simultanée de feuilletages plongés dans un même espace. La référence classique est le livre de Bol et Blaschke \cite{BB}. On renvoie au livre de J.V. Pereira et L. Pirio \cite{webg} pour une présentation moderne du sujet et plus de détails. On se limite dans cet article aux tissus implicites définis par A. Hénaut dans \cite{henaut}.\\


On considère une {\bf équation différentielle du premier ordre polynomiale de degré $d$ à coefficients analytiques}, définie par
\begin{equation}
F(x,y,y')=a_0(x,y)(y')^d+a_1(x,y)(y')^{d-1}+\cdots +a_d(x,y)=0,
\end{equation}
où $y'=\frac{dy}{dx}$ et où les coefficients $a_i\in \mathbb{C}\lbrace x,y \rbrace$ pour tout $i\in \lbrace 1,...,d \rbrace$ (anneau des fonctions analytiques en $x$ et $y$). On appelle {\bf polynôme de présentation} de $F$ et on note $P_F$, le polynôme en la variable $z$ et de paramètres $x$ et $y$ associé au tissu $F$ et défini par:
\begin{equation}
P_F(z;x,y)=a_0(x,y)z^d+a_1(x,y)z^{d-1}+\cdots +a_d(x,y).
\end{equation}
On a $F(x,y,y')=P_F (y';x,y)$. Le $z$-discriminant de ce polynôme, noté $\Delta$ est défini par:
\begin{equation}
\Delta=a_0^{2d-2}\underset{1\leq i < j \leq d}{\prod}(p_i(x,y)-p_j(x,y))^2,
\end{equation} 
où les $p_1(x,y), \ldots, p_d(x,y)$ sont les racines de $P_F(z;x,y)$ (voir \cite[p. 403]{gelfand}) et le résultant est donné par :
$R_{P_F}:=Result(P_F,\partial_{z}P_F)=(-1)^{\frac{d(d-1)}{2}}a_0 \Delta$. On supposera dans la suite que l'on se place hors du lieu des singularités, c'est à dire $R_{P_F} \neq 0$, appelée {\bf condition de factorisation}. En effet, le polynôme $P_F$ admet alors $d$ racines distinctes et on peut dans ce cas factoriser le polynôme $P_F$ par rapport à la variable $z$ sous la forme
$P_F(z;x,y)=a_0(x,y)\underset{i=1}{\overset{d}{\prod}}(z-p_i(x,y)).$. On déduit de cette construction la forme factorisée de l'équation différentielle initiale:
\begin{equation}
\label{fac}
F(x,y,y')=a_0(x,y)\underset{i=1}{\overset{d}{\prod}} \Delta_i (x,y,y'),
\end{equation}
où $\Delta_i (x,y,y')=y'-p_i(x,y)$, qu'on appellera {\bf forme préparée} de $F$.\\

Chaque équation différentielle $\Delta_i (x,y,y')=0$ induit un feuilletage $F_i$ du plan $(x,y)$. La forme préparée définie donc naturellement $d$-familles de feuilletages $F_i$ en position générale en chaque point $(x,y)\in \C^2$ associé au système différentiel $\mathscr{S}=\{ \Delta_i (x,y,y')=0,\ i=1,\dots ,d \}$, ce qui constitue un $\mathbf{d}${\bf -tissu} (voir \cite{webg}, Section 1.1 p.15). On a donc la définition suivante:

\begin{dfn}[$d$-tissu implicite \cite{henaut}] On appelle $d$-tissu implicite associé à une équation différentielle polynomial $F(x,y,y')=0$ de degré $d$ du premier ordre à coefficients analytiques satisfaisant la condition de factorisation \eqref{fac}, le tissu noté $\mathscr{W}_F ( F_1 ,\dots ,F_d )$ où chaque feuille $F_i$ est associée à l'équation différentielle $\Delta_i (x,y,y')=0$,
\end{dfn}






De manière réciproque, la donnée d'un $d$-tissu explicite $\mathcal{W}(F_1,\ldots, F_d)$, où les $F_i$ sont des feuilletages en position générale, permet de définir un $d$-tissu implicite. Il suffit de considérer la forme préparée définie par les racines :
\begin{equation}
p_i (x,y)=-\frac{\partial_x(F_i)}{\partial_y(F_i)} (x,y).
\end{equation}

L'intérêt des tissus {\bf implicites} est de posséder une représentation {\bf globale} (l'équation différentielle). L'idée est donc de voir si des propriétés de cet objet global permettent de {\bf caractériser certaines propriétés des tissus}.

\section{Symétrie des tissus implicites}\label{14}

On caractérise les groupes de symétrie des tissus implicites en utilisant la théorie classique sur les équations différentielles invariantes sous l'action d'un groupe de Lie telle que présentée dans le livre de P-J. Olver \cite{olver1}. On retrouve ainsi (Théorème \ref{theorem5} et Théorème \ref{normalcarac}) des résultats énoncés par A. Hénaut (\cite{henaut}, (LS) section 2 p.118 et p.119). 

\subsection{Rappels sur les groupe de symétrie d'une équation différentielle}
\label{13}

On rappelle rapidement la notion de groupe de symétrie en suivant le livre de P-J. Olver \cite{olver1} directement appliqué aux équations différentielles d'ordre 1.

\subsubsection{Groupes de symétries et générateurs infinitésimaux}
Soit une équation différentielle d'ordre $1$ de la forme $\Delta(x,y,y') = y'-p(x,y)=0$ avec $y'=\frac{dy}{dx}$, $x,\ y \in \mathbb{C}$. 

\begin{dfn}
Soit $\mathscr{S}$ une équation différentielle. Un groupe de symétries de $\mathscr{S}$ est un groupe local de transformations $G$ agissant sur un sous-ensemble ouvert $M$ du produit cartésien $\mathbb{C}^2$ tel que si  $y:M \rightarrow \mathbb{C}$ est une solution de $\mathscr{S}$ et dès que $g\cdot y$ est définie pour $g\in G$, alors $g \cdot y(x)$ est aussi une solution du système.
\end{dfn}

Les groupes de transformations que l'on considère dépendent d'un paramètre $\varepsilon$ tel que :  
\begin{align*}
G_\varepsilon : x \in M \mapsto g_\varepsilon\cdot (x,f(x)) = (\gamma_\varepsilon(x,f(x)), \phi_\varepsilon(x,f(x)).
\end{align*}

On peut voir l'action d'un élément du groupe comme l'action d'un champ de vecteurs $X$ sur $\mathbb{R}^2$ par dérivation par rapport à ce paramètre $\varepsilon$ :
\begin{align*}
X=\frac{d (g_\varepsilon \cdot x)}{d \varepsilon}\vert_{\varepsilon=0} \partial_x+\frac{d (g_\varepsilon \cdot f(x))}{d \varepsilon}\vert_{\varepsilon=0} \partial_y = \alpha_1(x,y) \partial_x+\alpha_2(x,y) \partial_y.
\end{align*} 
Le champ de vecteurs $X$ est appelé {\bf générateur infinitésimal} du groupe de symétrie $G$.\\
Si $g_\varepsilon$ transforme $y$ solution du système en $\Tilde{y}$ une autre solution, on cherche à déterminer l'action $g_\varepsilon$ sur $y'$.
Cette transformation est appelée \textbf{prolongement} et nous allons la caractériser sur le champs de vecteurs $X$.

Comme les transformations sont tangentes à l'identité, on a $\Tilde{x} = x + \varepsilon \alpha_1(x,y)+\mathcal{O}(\varepsilon^2)$ et $\Tilde{y} = y + \varepsilon \alpha_2(x,y)+\mathcal{O}(\varepsilon^2)$. En notant $D_x$ la différentielle totale par rapport à $x$, on a : 
\begin{align*}
    \frac{d\Tilde{y}}{d \Tilde{x}} = \frac{D_x(\Tilde{y})}{D_x(\Tilde{x})} 
    & = y'+ \varepsilon ( D_x(\alpha_2(x,y(x)))-y'D_x(\alpha_1(x,y(x)))) + \mathcal{O}(\varepsilon ^2) \\
    &= y' + \varepsilon \eta + \mathcal{O}(\varepsilon^2).
\end{align*}

En développant l'expression de $\eta$, on obtient :
\begin{align*}
    \eta = \partial_x(\alpha_2(x,y))+y'(\partial_y(\alpha_2(x,y)- \partial_x(\alpha_1(x,y))-(y')^2 \partial_y(\alpha_1(x,y)).
\end{align*}

Ce terme nous permet de définir le prolongateur de $X$ à l'ordre $1$ :
\begin{dfn}
    On appelle prolongateur à l'ordre $1$ du champs de vecteurs $X=\alpha_1(x,y) \partial_x+\alpha_2(x,y) \partial_y$ le champs $\Tilde{X}$:
\begin{align*}
        \Tilde{X} = \alpha_1 \partial_x + \alpha_2 \partial_y + \left(\partial_x(\alpha_2) +y'(\partial_y(\alpha_2)- \partial_x(\alpha_1)) - (y')^2 \partial_y(\alpha_1)  \right)\partial_{y'}.
\end{align*}
\end{dfn}

L'étude des $d$-tissus implicites fait intervenir seulement des équations différentielles d'ordre $1$ de la forme $y'=p(x,y)$. On a:

\begin{theorem}\label{dim1}
Un champ de vecteurs $X=\alpha_1 \partial_x + \alpha_2 \partial_y$ est un générateur infinitésimal d'un groupe de symétries $G$ d'une équation différentielle de la forme $y'(x)=p(x,y(x))$ si et seulement si les fonctions $\alpha_1$ et $\alpha_2$ en les variables $x$ et $y$ sont solutions de:
\begin{align*}
-\alpha_1\partial_x(p)-\alpha_2\partial_y(p)+\partial_x(\alpha_2)+(\partial_y(\alpha_2)-\partial_x(\alpha_1))p-\partial_y(\alpha_1)p^2=0.
\end{align*}
\end{theorem}

Ce théorème est un cas particulier explicite du critère d'invariance général formulé par P-J. Olver \cite[Theorem 2.31, p.104]{olver1} dans le cas d'une équation différentielle d'ordre $1$.\\

Le calcul explicite des groupes de symétrie dépend fortement de la forme du système d'équations. Deux équations différentielles sont dites équivalentes s'il existe un changement de variables qui transforme l'une en l'autre. Une question naturelle est donc de savoir si les groupes de symétries calculés pour des équations différentielles équivalentes sont isomorphes. C'est le cas (voir \cite[Proposition 6.13, p.185]{olver2}) et nous utiliserons ce résultat à plusieurs reprises pour simplifier nos calculs.

\begin{prop}
\label{isogroupe}
Deux équations différentielles équivalentes ont des groupes de symétries isomorphes.
\end{prop}

L'ensemble des générateurs infinitésimaux d'un groupe de symétrie possède une structure classique d'algèbre de Lie (voir \cite[Corollary 2.40, p.115]{olver1}) dont la dimension est un invariant.

\subsection{Caractérisation des symétries d'un tissu implicite}

On commence par définir le groupe de symétrie d'un tissu du plan (voir \cite{henaut}):
\begin{dfn}
Soit $\mathscr{W}(F_1,...,F_d)$ un $d$-tissu. Un groupe $G$ est un groupe de symétries du tissu si et seulement si il laisse globalement invariant le tissu.
\end{dfn}

Le Lemme suivant permet le transfert de l'étude du feuilletage à celui de la famille d'équations différentielles associées. Précisément, en utilisant le critère d'invariance formulé par P-J. Olver (\cite{olver1}, Theorem 2.71 p.161) on a le résultat suivant:

\begin{lem}
\label{critere}
Soit $\mathscr{W}_F (F_1 ,\dots ,F_d)$ le $d$-tissu implicite défini par $F$. Alors $G$ est un groupe de symétrie de $\mathscr{W}_F (F_1 ,\dots ,F_d )$ si et seulement si $G$ est un groupe de symétrie de chacune des équations différentielles $\Delta_i (x,y,y')=0$ pour $i=1,\dots ,d$.
\end{lem}

En utilisant le lemme \ref{critere} et le théorème \ref{dim1} on obtient donc la carcatérisation suivante des symétries d'un $d$-tissu:

\begin{theorem}[Groupe de symétries d'un $d$-tissu]
\label{theorem5}
Soit $\mathscr{W}_F (F_1,...,F_d)$ le $d$-tissu implicite associé à $F$. Soit $G$ un groupe de symétries de $\mathcal{W}_F$, alors tout générateur infinitésimal $X=\alpha_1(x,y)\partial_x+\alpha_2(x,y)\partial_y$ de $G$ satisfait le système d'équations:
\begin{equation} 
\label{eq}
-\alpha_1\partial_x(p_i)-\alpha_2\partial_y(p_i)+\partial_x(\alpha_2)+(\partial_y(\alpha_2)-\partial_x(\alpha_1))p_i-\partial_y(\alpha_1)p_i^2=0,
\end{equation}
où $p_i (x,y)=\frac{\partial_x F_i}{\partial_y F_i}$, $i=1,\dots ,d$.
\end{theorem}

Ce système peut se mettre sous forme "normale":\\

On note $V_{i:j}$ les matrices de type Vandermonde définies par $V_{i:j}=\left ( 
\begin{array}{ccc}
1 & p_i & p_i^2 \\
1 & \vdots & \vdots  \\
1 & p_j & p_j^2 
\end{array}
\right )$ et $M_{i:j}$ les matrices définies par $M_{i:j}=\left ( 
\begin{array}{cc}
\partial_x p_i & \partial_y p_i \\
\vdots & \vdots \\
\partial_x p_j & \partial_y p_j 
\end{array}
\right ) $. 
On note $\mathbf{\alpha} =(\alpha_1 ,\alpha_2 )$. L'écriture matricielle du système \eqref{eq} est donnée par
\begin{equation}
\left \{
\begin{array}{l}
-M_{1:3} \mathbf{\alpha} + V_{1:3} \left ( 
\begin{array}{c}
-\partial_x \alpha_2 \\
\partial_x \alpha_1 -\partial_y \alpha_2 \\
\partial_y \alpha_1
\end{array}
\right ) 
=0 ,\\
-M_{4:d} \mathbf{\alpha} + V_{4:d} \left ( 
\begin{array}{c}
-\partial_x \alpha_2 \\
\partial_x \alpha_1 -\partial_y \alpha_2 \\
\partial_y \alpha_1
\end{array}
\right ) 
=0 .
\end{array}
\right . 
\end{equation}

La matrice $V_{1:3}$ est une matrice de Vandermonde qui est inversible car $p_1$, $p_2$ et $p_3$ sont deux à deux distincts. Le système est donc équivalent à la forme suivante: 

\begin{equation}
\left \{
\begin{array}{l}
-V_{1:3}^{-1} M_{1:3} \mathbf{\alpha} + \left ( 
\begin{array}{c}
-\partial_x \alpha_2 \\
\partial_x \alpha_1 -\partial_y \alpha_2 \\
\partial_y \alpha_1
\end{array}
\right ) 
=0 ,\\
C_d  \mathbf{\alpha}  =0 ,
\end{array}
\right . 
\end{equation}
où 
\begin{equation}
C_d = -M_{4:d} +V_{4:d} V_{1:3}^{-1} M_{1:3} ,
\end{equation}
est appelée {\bf matrice de compatibilité}.\\

On obtient donc le système noté $(\mathscr{S})$ par A. Henaut:

\begin{theorem}[Equation normalisée du groupe de symétries d'un tissu implicite]
\label{normalcarac}
Pour $d\geq 3$, le système d'équations \ref{eq} peut s'écrire sous une forme normalisée:
 \begin{align*}
\tag{$\star_d$}
\left\{\begin{array}{ccc}
-\partial_x( \alpha_2) &+g_d \alpha_1+h_d \alpha_2=0 \\
\partial_x( \alpha_1)-\partial_y(\alpha_2)&+g_{d-1}\alpha_1+h_{d-1}\alpha_2 =0 \\
\partial_y(\alpha_1)&+g_{d-2}\alpha_1+h_{d-2}\alpha_2=0 \\
&g_{d-3}\alpha_1+h_{d-3}\alpha_2=0 \\
&\vdots\\
&g_1 \alpha_1+h_1 \alpha_2=0.
\end{array}\right.
\end{align*}
où les coefficients $g_i$ et $h_i$, pour $i=1,...,d$, dépendent des pentes $p_i$.
\end{theorem}

On peut bien entendu expliciter les coefficients $g_i$ et $h_i$. Par exemple, on obtient:
\begin{equation}
\left .
\begin{array}{l}
g_1 = - \di\frac{p_2 p_3}{(p_3 -p_1 ) (p_2 -p_1)} \partial_x p_1 +
\di \frac{p_1 p_3}{(p_3 -p_2)(p_2 -p_1)} \partial_x p_2 
- \di\frac{p_1 p_2}{(p_3 -p_1) (p_3 -p_2 )} \partial_x p_3 ,\\
g_2 = \di\frac{p_2 +p_3}{(p_3 -p_1 ) (p_2 -p_1)} \partial_x p_1 
-\di \frac{p_1 +p_3}{(p_3 -p_2)(p_2 -p_1)} \partial_x p_2 
+ \di\frac{p_1 +p_2}{(p_3 -p_1) (p_3 -p_2 )} \partial_x p_3 ,\\
g_3 = - \di\frac{1}{(p_3 -p_1 ) (p_2 -p_1)} \partial_x p_1 +
\di \frac{1}{(p_3 -p_2)(p_2 -p_1)} \partial_x p_2 
- \di\frac{1}{(p_3 -p_1) (p_3 -p_2 )} \partial_x p_3 ,
\end{array}
\right .
\end{equation}
et des expressions analogues pour $h_1$, $h_2$ et $h_3$ en remplacant les $\partial_x p_i$ par des $\partial_y p_i$. \\

On calcule de manière analogue les autres coefficients.\\

Le système ainsi formé est par ailleurs de la même forme que celui obtenu par A. Henaut (\cite{henaut2}, équation ($\star_d$) p.433) caractérisant les {\bf relations abéliennes} d'un tissu.\\

Les relations $C_d \mathbf{\alpha}  =0$ sont appelées {\bf relations de compatibilité} par A. Hénaut.\\

L'étude des solutions de ce système peut s'effectuer dans le cadre du théorème de Cauchy-Kovaleskaya tel qu'exposé dans (\cite{olver1}, Chap.2, p.162) dont on reprend les notations. Pour $\mathbf{\alpha}=(\alpha_1 ,\alpha_2 )$, on note $\mathbf{\alpha}^{(1)} = (\alpha_1 ,\alpha_2 ,\partial_x \alpha_1 ,\partial_x \alpha_2 ,\partial_y \alpha_1 ,\partial_y \alpha_2 )$. Le système ($\star_d$) est donc équivalent à $d$ équations $\Delta_1 (x,\mathbf{\alpha}^{(1)} )=0,\dots, \Delta_d (x,\mathbf{\alpha}^{(1)} )=0$. Pour $d\geq 3$, on obtient un système sur-déterminé au sens de P.J. Olver (\cite{olver1}, Definition 2.86,p.170-171). Les conditions de compatibilité ainsi que la troisième relation peuvent être vue comme des {\bf conditions d'intégrabilité} par rapport aux solutions du système 
\begin{equation}
\left\{
\begin{array}{ccc}
-\partial_x( \alpha_2) &+g_d \alpha_1+h_d \alpha_2=0 ,\\
\partial_x( \alpha_1)-\partial_y(\alpha_2)&+g_{d-1}\alpha_1+h_{d-1}\alpha_2 =0 
\end{array}
\right .
\end{equation}

Ces conditions vont compromettre la résolubilité des systèmes ($\star_d$) pour $d\geq 3$. En particulier, on a:

\begin{prop} 
Pour $d\geq 3$, le système \eqref{eq} n'admet génériquement pas de solutions.
\end{prop}

Une autre façon de formuler le résultat précédent est: {\it Pour $d\geq 3$, un $d$-tissu générique n'admet pas de symétries}.\\

Dans les sections suivantes, nous donnons des exemples de résolutions explicites du système ($\star_d$) pour des tissus donnés. 

\section{Exemples de Groupes de symétries de tissus}
\label{sectionexemple}

Nous calculons explicitement l'algèbre de Lie du groupe de symétries de $d$-tissus. Les cas $d=1$ et $d=2$ sont triviaux et correspondent à des algèbres de Lie de dimensions infinies. Dans le cas $d\geq 3$, les algèbres de Lie obtenues sont de dimension $1$ ou $3$ illustrant ainsi le résultat général de A. Hénaut (\cite{henaut},Proposition 1 p.124) montrant que l'algèbre de Lie des groupes de symétries d'un tissu implicite est toujours de dimension $0$, $1$ ou $3$. 

\subsection{Cas des $1-$ et $2-$tissus}

Pour les tissus du plan, on voit le cas particulier des 1-tissus et 2-tissus dont les algèbres de Lie des symétries sont de dimension infinie:

\begin{lem}[1-tissu]L'algèbre de Lie des symétries d'un 1-tissu est de dimension infinie, chaque symétrie est de la forme $X=\alpha_1(x,y)\partial_x+\alpha_2(y)\partial_y$, avec $\alpha_1 \in \mathbb{C}\lbrace x,y \rbrace$ et $\alpha_2\in \mathbb{C}\lbrace y \rbrace$.
\end{lem}
\begin{proof}[Démonstration]
Pour un 1-tissu, on peut supposer, quitte à faire un changement de variables, que  $p_1=0$. Ainsi le système \eqref{eq} se réduit à $\partial_x(\alpha_2)=0$ et aucune condition n'apparaît sur $\alpha_1$. L'algèbre de Lie est de dimension infinie et engendrée par: $X=\alpha_1(x,y)\partial_x+\alpha_2(y)\partial_y$ où $\alpha_2(y)$ est analytique en $y$. La proposition \ref{isogroupe} termine la démonstration.
\end{proof}

\begin{lem}[2-tissu]
L'algèbre de Lie des symétries d'un 2-tissu est de dimension infinie.
\end{lem}

\begin{proof}[Démonstration]
Pour $d=2$, avec $p_1 \neq p_2$ on peut supposer, quitte à faire un changement de variables "redressant" les feuilles, que $p_1=0$ et $p_2=1$. On obtient les deux équations $\partial_x(\alpha_2)=0$ pour $p_1=0$ et $-\partial_x(\alpha_2)+(\partial_x(\alpha_1)-\partial_y(\alpha_2))+\partial_y(\alpha_1)=0$ pour $p_2=1$. Comme pour le cas précédent, on a que $\alpha_2$ est une fonction de la variable $y$. La seconde équation est $\partial_x(\alpha_1)+\partial_y(\alpha_1)=\partial_y(\alpha_2)$. Si les solutions ont la forme générale $\alpha_1=\underset{j,k \geq 0}{\sum}p_{j,k}x^jy^k$ et $\alpha_2=\underset{j\geq 0}{\sum}q_jy^j$ alors la dernière équation implique $\forall i \geq 1$, $\forall j\geq 0$, $p_{i+1,j}(i+1)+p_{i,j+1}(j+1)=0$ et si $i=0$ alors $\forall j \geq 1$, $p_{1,j}+p_{0,j+1}(j+1)=q_{j+1}(j+1)$. Ces dernières relations définissent les relations sur les coefficients des symétries d'un 2-tissu. La proposition \ref{isogroupe} termine la démonstration.
\end{proof}

\subsection{Le cas parallèle}

Les tissus parallèles, donnés par des feuilletages de droites parallèles, jouent un rôle particulier dans l'étude des tissus.

\begin{dfn}[Tissu parallèle] Un $d$-tissu est dit {\it parallèle} s'il est donné par la superposition de $d$ pinceaux de droites en position générale, écrit sous la forme $\mathcal{W}(a_1x-b_1y,...,a_dx-b_dy)$ où $(a_i,b_i) \in \C^2 \setminus \lbrace 0 \rbrace$. L'hypothèse de position générale est équivalente à $p_i\neq p_j, \ \forall i\neq j$ où $p_i=-\frac{a_i}{b_i}$, $i=1,\ldots, d$, représentent les pentes du pinceau.
\end{dfn}

La terminologie de tissu parallèle vient du fait que chaque feuilletage est constitué de droites parallèles. Les symétries d'un tel tissu sont données par le Lemme suivant:

\begin{lem}
\label{casparallel}
Soit un $d$-tissu parallèle, avec $d\geq 3$, défini par les pentes constantes $p_i(x,y)$ $i=1,...,d$. Alors son algèbre de Lie des symétries est engendrée par les trois générateurs infinitésimaux $\mathfrak{g}=\lbrace \partial_x, \ \partial_y, \ x\partial_x+y\partial_y \rbrace$.
\end{lem}

\begin{proof}[Démonstration]
Considérons un $d$-tissu parallèle avec $d\geq 3$, donné par des pentes constantes $p_i(x,y)$. On a donc par définition $M_{1:3}=0$ et $M_{4:d} =0$ car ces matrices ne dépendent que des dérivées des pentes, d'où $C_d =0$. Le système ($\star_d$) se réduit donc à:
\begin{align*}
-\partial_x( \alpha_2) =0,\ \ 
\partial_x( \alpha_1)-\partial_y(\alpha_2)=0, \ \
\partial_y(\alpha_1)=0.
\end{align*}
La fonction $\alpha_2$ ne dépend que de $y$ et $\alpha_1$ de $x$. On cherche des solutions du système dans la classe polynomiale. On pose $\alpha_1(x)=\underset{j=0}{\overset{k}{\sum}}r_jx^j$ et  $\alpha_2(y)=\underset{j=0}{\overset{m}{\sum}}q_jy^i$. On obtient $\alpha_1=r_0+xr_1$ et $\alpha_2=q_0+yq_1$ avec $q_1 =r_1$. L'algèbre de Lie étant au maximum de dimension $3$, on déduit que les symétries sont engendrées par $\partial_x$, $\partial_y$ et $x\partial_x +y \partial_y$.
\end{proof}

\subsection{Un 3-tissu donné par \'Elie Cartan}

Cet exemple est donné par A.Hénaut dans \cite{henaut}. On considère le $3$-tissu donné par son polynôme de présentation $P_F(z;x,y)=(z^2-1)(z-u(x))$ où $u$ est une fonction analytique de $x$. 

\begin{lem}
L'algèbre de Lie des symétries du tissu de Cartan est donnée par
$\mathfrak{g}=\lbrace \partial_y \rbrace$.
\end{lem}

\begin{proof}[Démonstration]
Par définition du tissu, les pentes sont données par $p_1=1$, $p_2=-1$ et $p_3=u(x)$. Une symétrie $X=\alpha_1(x,y)\partial_x+\alpha_2(x,y)$ de ce 3-tissu doit satisfaire :
\begin{align*}
\partial_x(\alpha_2)=\partial_y(\alpha_1), \ \
\partial_x(\alpha_1)=\partial_y(\alpha_2), \ \
(u^2-1)\partial_y(\alpha_1)+\partial_x(u)\alpha_1=0.
\end{align*}

Les deux premières équations du système donnent : 
\begin{align*}
\left(B(x,y)^2+\partial_x \left(B(x,y)-\left(\frac{u'(x)}{u(x)^2-1}\right)^2\right)\right)\alpha_1=0.
\end{align*}
et les dérivées d'ordre 1 et 2 de  $\alpha_1$ donnent :
\begin{align*}
&\partial_x(\alpha_1)=B(x,y)\alpha_1, \ \ \partial^2_{x,x}(\alpha_1)=(B^2(x,y)+\partial_x(B(x,y))\alpha_1, \\
& \partial^2_{y,y}(\alpha_1)=\left( \frac{u'(x)}{u(x)^2-1}\right)^2\alpha_1,
\end{align*}
où $B(x,y)=\frac{c'(x)}{c(x)}-\left(\frac{u'(x)}{u(x)^2-1}\right)'y$.
On obtient finalement :
\begin{align*}
\left(B(x,y)^2+\partial_x \left(B(x,y)-\left(\frac{u'(x)}{u(x)^2-1}\right)^2\right)\right)\alpha_1=0.
\end{align*}
Pour une fonction $u=u(x)$ générique, on a $\alpha_1=0$ et $\partial_y(\alpha_2)=\partial_x(\alpha_2)=0$, alors $\alpha_2$ est une constante.
\end{proof}

\subsection{Le tissu de Muzsnay} 
\label{muzs}

Dans \cite{muzsnay}, suite à une série d'articles dans l'optique de la résolution de la conjecture de Blaschke sur la linéarisation des 3-tissus dans $\C^2$, Z. Muzsnay considère le 3-tissu $\mathcal{W}_M$ donné par les trois feuilletages suivant:
\begin{equation}
F_1(x,y):=x, \ F_2(x,y):=y \text{ et } F_3(x,y):=(x+y)e^{-x}.
\end{equation}
Ce tissu a fait l'objet de nombreux articles concernant sa linéarisation notamment \cite{goly} et \cite{grim}. Les différents auteurs ont donné des caractérisations des tissus linéarisables via des méthodes différentes et ne donnent pas le même résultat. Dans \cite{muzsnay}, l'auteur pour clore la controverse exhibe un biholomorphisme de linéarisation explicite. Nous proposons ici le calcul de son groupe de symétrie.\\

Le calcul des pentes du tissu fait apparaître des pentes infinies. En effectuant le changement de variables $(x,y) \mapsto (x+y, y-x)=(u,v)$, on obtient ;

\begin{lem}[Forme préparée du tissu] Le tissu $\mathcal{W}_M$ est biholomorphiquement conjugué au tissu donné par les trois feuilletages définis par $G_1(x,y):=\frac{u-v}{2}$, $G_2(x,y):=\frac{u+v}{2}$ et $G_3(x,y):=ue^{-\frac{u-v}{2}}$.
\end{lem}

On peut alors calculer le groupe de symétrie du tissu dans ce nouveau système de coordonnées:

\begin{lem}
L'algèbre de Lie des symétries du tissu $\mathcal{W}_M$ mis sous forme préparée est donnée par $\mathfrak{g}=\lbrace \partial_v \rbrace$ où $c$ est une fonction ne dépendant que de $u$.
\end{lem}

\begin{proof}[Démonstration]
Le calcul des trois pentes associées aux feuilletages donne $p_1=1$, $p_2=-1$ et $p_3=\frac{u-2}{u}$. En introduisant ces pentes dans les équations définissant les symétries, on obtient le système: $\partial_u\alpha_2=\partial_v\alpha_1$, $\partial_v\alpha_2=\partial_u\alpha_1$ et $-2 \alpha_1+(4u-4)\partial_v\alpha_1=0$.

La troisième équation s'intégre directement et donne $\alpha_1(u,v)=c(u)e^{\frac{v}{2(u-1)}}$ où $c(u)$ est une fonction qui dépend de $u$. En réintroduisant dans les deux premières équations, on obtient : $\alpha_2=\alpha_1 \times \left( 2-\frac{v}{u-1}\right)+2c'(u)(u-1)e^{\frac{v}{2(u-1)}}+c$ où $c$ est une constante. En procédant comme pour le tissu de Clairaut, on obtient l'équation différentielle $\alpha_1 \times \left(\frac{1}{4(u-1)^2}- \partial_u \beta- \beta^2 \right)=0$ avec $\beta=\frac{c'(u)}{c(u)}-\frac{v}{2(u-1)^2}$. Donc soit $\alpha_1=0$ soit le second facteur est nul ce qui est impossible par rapport aux degrés en $u$ et $v$. Donc $\alpha_1=0$ et on récupère une unique symétrie $c \partial_v$.
\end{proof}

On peut conclure en vertu de la Proposition \ref{isogroupe} que l'algèbre de Lie des symétries du tissu sous forme préparée étant isomorphe à celle de $\mathcal{W}_M$, cette dernière ne compte qu'une unique symétrie.

\subsection{Tissu de Clairaut}
\label{regularite}

Le {\bf tissu de Clairaut} étudié par A. Hénaut \cite{henaut} est un 3-tissu défini par les submersions $F_1(x,y):=y-x$, $F_2(x,y)=:=x+y$ et $F_3(x,y):=\frac{y}{x}$. Les trois pentes associées sont $p_1=1$, $p_2=1$ et $p_3=\frac{y}{x}$. Le groupe de symétrie est donné par (voir \cite{henaut}):

\begin{lem}
L'algèbre de Lie des symétries du 3-tissu de Clairaut défini par les submersions $F_1=y-x$, $F_2=y+x$ et $F_3=\frac{y}{x}$ est de dimension 3 et engendrée par les champs de vecteurs:
\begin{align*}
 & x\partial_x+y\partial_y, \ y\partial_x+x\partial_y, \\
 &\ \left( x \ell n(\left(\bigg| x^2-y^2 \bigg|\right) +y \ell n\left(\bigg|\frac{x+y}{x-y}\bigg|\right) \right)\partial_x + \left( y \ell n\left(\bigg| x^2-y^2 \bigg|\right) +x \ell n\left(\bigg|\frac{x+y}{x-y}\bigg|\right)\right) \partial_y .
\end{align*}
\end{lem}

Nous donnons ici une démonstration détaillée de ce résultat.

\begin{proof}[Démonstration]
D'après le théorème \ref{theorem5} et par simplication du système, les symétries vérifient:\\
$$\partial_x(\alpha_2)=\partial_y(\alpha_1), \ \ \partial_y(\alpha_2)=\partial_x(\alpha_1), \ \ 
-x^2\partial_x(\alpha_2)+y^2\partial_y(\alpha_1)-y\alpha_1+x\alpha_2=0.$$
Les deux premières lignes donnent $\partial^2_{x,x}(\alpha_1)-\partial^2_{y,y}(\alpha_1)=0$ et $\partial^2_{x,x}(\alpha_2)-\partial^2_{y,y}(\alpha_2)=0$.
Les symétries dans la classe polynomiale sont engendrées par les champs $X_1 =x\partial_x+y\partial_y$ et $X_2=y\partial_x+x\partial_y$.

Pour élargir la classe des symétries, on peut utiliser un résultat classique sur les équations aux dérivées partielles du second ordre qui donne la forme générale d'une solution :
\begin{align*}
\alpha_i=f_i(x+y)+g_i(x-y), \ i=1,2,
\end{align*}
où $f$ et $g$ sont analytiques. 
Les deux premières équations de symétries impliquent $f_1=f_2$ et $g_1=g_2$. On a donc $\alpha_1$ et $\alpha_2$ de la forme $\alpha_1 (x,y)=f(x+y)+g(x-y)$ et $\alpha_2 (x,y)=f(x+y)-g(x-y)$. La troisième équation devient:
\begin{align*}
(y-x)\left( (x+y)f'(x+y)-f(x+y)\right)+(x+y)\left((x-y)g'(x-y)-g(x-y)\right)=0.
\end{align*}
Si $y=x$, on obtient $xg(0)=0$ donc $g(0)=0$ et $g$ est de la forme $g(v)=vG(v)$ avec $v=x-y$ et $G$ est une fonction de $v$ sans terme constant. De la même manière en prenant $y=-x$ on obtient $f(t)=tF(t)$ avec $t=x+y$ et $F$ une fonction de la variable $t$. On a donc:
\begin{align*}
\alpha_1 (x,y)=(x+y)F(x+y)+(x-y)G(x-y), \ \ \alpha_2=(x+y)F(x+y)-(x-y)G(x-y).
\end{align*}
En remplaçant dans la deuxième équation et en posant $t=x+y$ et $v=x-y$, on a:
\begin{align*}
tF'(t)+vG'(v)=0,
\end{align*}
ainsi $tF'(t)=c$ et $vG'(v)=-c$ où $c$ est une constante dans $\C$. En résolvant ces équations différentielles $F(t)=c \ \ell n(\vert t \vert)$ \text{ et } $G(v)=-c \ \ell n(\vert v \vert)$. Finalement on obtient:
$$ \alpha_1 = x \ell n(\left(\bigg| x^2-y^2 \bigg|\right) +y \ell n\left(\bigg|\frac{x+y}{x-y}\bigg|\right) \text{ et } \alpha_2= y \ell n\left(\bigg| x^2-y^2 \bigg|\right) +x \ell n\left(\bigg|\frac{x+y}{x-y}\bigg|\right) .$$
Ce qui termine la démonstration.
\end{proof}

\subsection{Le tissu de Zariski}
\label{zariski}

Le tissu de Zariski est défini implicitement par $F(x,y,p)=p^3+x^my^n$ avec $m$ et $n$ dans $\N$. Ce tissu admet la factorisation suivante:
\begin{align*}
F(x,y,p)=(p+x^{\frac{m}{3}}y^{\frac{n}{3}})(p-e^{\frac{i \pi}{3}}x^{\frac{m}{3}}y^{\frac{n}{3}})(p-e^{-\frac{i \pi}{3}}x^{\frac{m}{3}}y^{\frac{n}{3}}).
\end{align*}

Le Lemme suivant est donné sans démonstration dans (\cite[Exemple 3 p.14]{henaut}):

\begin{lem}
Le 3-tissu de Zariski défini implicitement par $F(x,y,p)=p^3+x^my^n$ admet les algèbres de Lie des symétries de dimension $3$ définies par:\\

$\bullet$ $\mathfrak{g}=\lbrace x^{-m/3} \partial_x,  y^{n/3} \partial_y, x (-n+3)\partial_x +y (3+m) \partial_y\rbrace$ si $n\not=3$,

$\bullet$ $\mathfrak{g}=\lbrace x^{-m/3} \partial_x,  y \partial_y, \frac{3 x}{m+3} \partial_x +y \ell n (y) \partial_y\rbrace$ si $n=3$.
\end{lem}

\begin{proof}[Démonstration]
Une symétrie $X=\alpha_1 \partial_x+\alpha_2\partial_y$ doit satisfaire le système composé des trois équations aux dérivées partielles suivantes après simplification : $\partial_x(\alpha_2)=0$, $\partial_y(\alpha_1)=0$ et $$(\partial_x(\alpha_1)-\partial_y(\alpha_2))x^{\frac{m}{3}}y^{\frac{n}{3}}+\frac{m}{3}x^{\frac{m}{3}-1}y^{\frac{n}{3}}\alpha_1+\frac{n}{3}x^{\frac{m}{3}}y^{\frac{n}{3}-1}\alpha_2=0.$$
Les deux premières conditions impliquent que $\alpha_1 =\alpha _1 (x)$ et $\alpha_2 =\alpha_2 (y)$. La dernière équation peut se mettre sous la forme : $\partial_x(\alpha_1)-\partial_y(\alpha_2)+\frac{m}{3x}\alpha_1+\frac{n}{3y}\alpha_2=0$ en dehors de $xy=0$. Pour déterminer $\alpha_1$, on doit résoudre $\partial_x(\alpha_1)+\frac{m}{3x}\alpha_1+ b(y)=0$ avec $b(y)=-\partial_y(\alpha_2)+\frac{n}{3y}\alpha_2$. De manière symétrique, on a $\partial_y(\alpha_2)-\frac{n}{3x}\alpha_2+ d(x)=0$ avec $d(x)=\partial_x(\alpha_1) +\frac{m}{3x}\alpha_1$. Cela revient à résoudre une équation différentielle de la forme 
$z'+a\frac{z}{t} +b=0$ pour $z(t)$. Les solutions pour $a\not=-1$ sont de la forme 
$z(t)=ct^{-a} +\frac{b}{a+1} t$ et pour $a=-1$ sont données par
$z(t)=-bt\ell n (t) +c t$. Pour $\alpha_1$, la condition $a=m/3\not =1$ est toujours satisfaite, on a donc $\alpha_1 (x)=\mu x^{-m/3} +\delta \frac{3x}{m+3}$, ce qui donne $d(x)=\delta$. Pour la composante $\alpha_2$, on doit distinguer suivant le cas $n\not=3$ i.e $a=-1$ ou $n=3$. Pour $n\not=3$, on obtient $\alpha_2 (y)= \gamma y^{n/3}+\frac{3\delta}{3-n} y$ et $\alpha_2 (y)= -\delta y \ell  (y) +\gamma y$ si $n=3$. On obtient donc des champs de la forme :
$$X=\mu (x^{-m/3} ,0) +\gamma (0,y^{n/3}) +\mu (x(-n+3),y(m+3))$$
si $n\not=3$ et $X=\mu (x^{-m/3} ,0) +\gamma (0,y^{n/3}) +\mu \left ( \frac{3x}{m+3},y \ell  (y) \right )$ sinon. 
\end{proof}

\subsection{Tissu hexagonal standard}

En suivant \cite{beau}, p.103, on introduit une classe particulière de 3-tissus dits $hexagonaux$. On va se placer dans $(\C^2,0)$ mais la construction reste valide pour $(\R^2,0)$.

\begin{dfn} Un 3-tissu dans $(\C^2,0)$  est dit hexagonal si tout hexagone suffisamment petit autour de 0 et passant par une feuille de chaque feuilletage est fermé.
\end{dfn}

Le 3-tissu hexagonal le plus simple est celui donné par les feuilletages $F_1(x,y)=x$, $F_2(x,y)=y$ et $F_3(x,y)=x+y$. Cependant, comme pour la section "Sur un tissu linéarisation", \ref{muzs}, le calcul des pentes fait encore intervenir une pente infinie. Nous allons donc présenter le tissu dans un nouveau système de coordonnées pour effectuer le calcul des symétries. 

\begin{lem}[Mise sous forme préparée]
Le tissu hexagonal donné par $F_1, \ F_2$ et $F_3$ est biholomorphiquement conjugué au tissu donné par les feuilletages $G_1(x,y)=x+y$, $G_2(x,y)=y-x$ et $G_3(x,y)=y$.
\end{lem}

\begin{proof}[Démonstration]
On utilise à nouveau le changement de variables $u=x+y$, $v=y-x$.
\end{proof}

On peut alors calculer le groupe de symétrie du tissu dans son nouveau système de coordonnées:

\begin{lem}
L'algèbre de Lie $\mathfrak{g}$ des symétries de ce tissu hexagonal défini par les submersions $G_1(x,y)=x+y$, $G_2(x,y)=y-x$ et $G_3(x,y)=y$ est donné par $
\mathfrak{g}=\lbrace \partial_x, \partial_y,x\partial_x+ y\partial_y \rbrace$.
\end{lem}

\begin{proof}[Démonstration]
Les trois pentes $p_1=-1$, $p_2=1$ et $p_3=0$, ainsi que les trois champs de vecteurs $X_1=\partial_x-\partial_y$, $X_2=\partial_x+\partial_y$ et $X_3=\partial_x$. Pour trouver les symétries de ce tissu, on doit résoudre le système suivant:
\begin{align*}
\left\{\begin{array}{cccc}
-\partial_x(\alpha_2)-\partial_x(\alpha_1)+\partial_y(\alpha_2)+\partial_y(\alpha_1)=0, \\
-\partial_x(\alpha_2)+\partial_x(\alpha_1)-\partial_y(\alpha_2)+\partial_y(\alpha_1)=0, \\
\partial_x(\alpha_2)=0.
\end{array}\right.
\end{align*}

La troisième équation implique que $\alpha_2=\alpha_2(y)$ est une fonction de la variable $y$. La résolution implique que $\alpha_1$ ne dépend que de la variable $x$. Enfin la condition $\partial_x(\alpha_1)=\partial_y(\alpha_2)$ implique $\alpha_1=p+qx$ et $\alpha_2=p'+qy$ où $p,p',q$ sont des constantes. En séparant en fonction de constantes indépendantes, on obtient les trois générateurs.
\end{proof}

\section{Symétries des tissus implicites et polynôme de Darboux}
\label{19}
Dans cette section après avoir fait des rappels sur les polynômes de Darboux, on passe en revue les liens entre symétries et module de dérivations associé à un tissu via la courbe discriminante du polynôme de présentation. Certains de ces résultats ont été obtenus par A. Hénaut dans \cite{henaut}.

\subsection{Autour du Théorème de Darboux}

Dans cette section, on rappelle des résultats classiques sur la théorie de Darboux sur l'intégrabilité de champs de vecteurs polynomiaux. On renvoie à \cite{yako} pour plus de détails. Dans la suite, on a $\K= \R$ ou $\C$.\\

Soit un champ de vecteurs $X=\underset{i=1}{\overset{d}{\sum}}f_i(x)\partial_{x_i}$ dans $\K^d$ avec $f_i(x) \in \K[x]$, $x=(x_1,...,x_d)$. 

\begin{dfn}
Une courbe algébrique $\mathscr{C}$ définie par $\lbrace g=0 \rbrace$ avec $g \in \K[x]$  est dite invariante par le champ de vecteurs $X$ s'il existe un polynôme $K \in \K[x]$, appelé cofacteur de $\mathscr{C}$, tel que $X(g)=K.g$.  Une courbe est appelée polynôme de Darboux pour le champ $X$ s'il existe un cofacteur.
\end{dfn}

\begin{theorem}
Soit $g\in \K[x]$ un polynôme écrit sous forme irréductible $g=g_1^{n_1}\cdots g_r^{n_r}$ dans $K[x]$. Le champ de vecteurs $X$ admet la courbe $\mathscr{C}=\lbrace g=0 \rbrace$ comme courbe invariante avec un cofacteur noté $K_g$ si et seulement si chaque courbe $\mathscr{C}_i=\lbrace g_i=0 \rbrace$ est invariante avec un cofacteur $K_{f_i}$ tel que $K_g=n_1 K_{g_1}+\cdots +n_r K_{g_r}$.
\end{theorem}

Le Théorème classique suivant, appelé Théorème de Darboux, donne une interprétation dynamique de l'invariance en particulier en terme d'intégrabilité grâce à l'existence de courbes algébriques invariantes.

\begin{theorem}[Théorème de Darboux] 
Soit un champ de vecteurs $X=P(x,y)\partial_x + Q(x,y)\partial_y$ où $P(x,y),Q(x,y) \in \K [x,y]$. Le degré du champ $X$ est $m=max(deg(P(x,y)),deg(Q(x,y))$. Supposons que $X$ admet $p$ courbes algébriques irréductibles invariantes $\lbrace g_i=0 \rbrace$ de cofacteurs $K_i$, $i=1,...,p$. \\
$i)$ Il existe $\lbrace \lambda_i \rbrace_{i=1,...,p} \subset \K ,  \lambda_i$ non tous nuls tels que $\underset{i=1}{\overset{p}{\sum}}\lambda_i K_i=0$ si et seulement si $g_1^{\lambda_1}...g_p^{\lambda_p}$ est une intégrale première de $X$. \\
$ii)$ Si $p \geq \frac{m(m+1)}{2}+1$, alors il existe $\lbrace \lambda_i \rbrace_{i=1,...,p} \subset \K ,  \lambda_i$ non tous nuls tels que $\underset{i=1}{\overset{p}{\sum}}\lambda_i K_i=0$.
\end{theorem}

\subsection{Tissus implicites et Polynôme de Darboux}

 On redémontre ici le résultat présent dans \cite{henaut} qui permet de faire le lien entre la symétrie d'un tissu et la courbe discriminante, qui s'avère être un polynôme de Darboux de la symétrie.\\

\begin{theorem}
\label{discinv}
Soit $\mathcal{W}(F)$ un $d$-tissu défini par le polynôme de présentation $P_F(z;x,y)=a_0(x,y)z^d+...+a_d(x,y)$. Si le champ de vecteurs $X=\alpha_1\partial_x+\alpha_2\partial_y$ est une symétrie du tissu $\mathcal{W}(F)$ alors le discriminant de $P_F$, noté $\Delta$, est un polynôme de Darboux de $X$.
\end{theorem}

\begin{proof}[Démonstration]
On reprend les notations introduites dans la section \ref{formepre}. Le discriminant est donné par $\Delta=a_0^{2d-2}\underset{1\leq i < j \leq d}{\prod} (p_i-p_j)^2=a_0^{2d-2}\underset{1\leq i < j \leq d}{\prod} \Delta_{i,j}^2$, avec $\Delta_{i,j}=p_i-p_j$. On a :
\begin{align*}
X.\Delta_{i,j}=X(p_i-p_j) =X(p_i)-X(p_j)=\alpha_1\partial_x(p_i)-\alpha_2\partial_y(p_i)-\alpha_1\partial_x(p_j)+\alpha_2\partial_y(p_j).
\end{align*}
Par invariance, en utilisant \eqref{eq}, on obtient: 
\begin{align*}
\alpha_1\partial_x(p_i)+\alpha_2\partial_y(p_i)=\partial_x(\alpha_2)+(\partial_y(\alpha_2)-\partial_x(\alpha_1))p_i-\partial_y(\alpha_1)p_i^2,
\end{align*}
et la même équation pour $j$. On a alors:
\begin{align*}
X(p_i-p_j)&=-(\partial_x(\alpha_1)-\partial_y(\alpha_2))(p_i-p_j)-\partial_y(\alpha_1)(p_i^2-p_j^2),\\
&=-\left [ \partial_x(\alpha_1)-\partial_y(\alpha_2)-\partial_y(\alpha_1)(p_i+p_j)\right ] (p_i-p_j), \\
&=\lambda_{i,j}\Delta_{i,j},
\end{align*}
où $\lambda_{i,j}=-(\partial_x(\alpha_1)-\partial_y(\alpha_2)-\partial_y(\alpha_1)(p_i+p_j)).$
Ainsi $\Delta_{i,j}$ est un polynôme de Darboux pour tout $\leq i<j\leq d$, et en utilisant le théorème de Darboux, on obtient que $\Delta$ est un polynôme de Darboux de cofacteur $\lambda=\underset{1\leq i < j \leq d}{\sum}\lambda_{i,j}$.
\end{proof}

\subsection{Module de dérivations et courbes discriminantes}

Les géomètres ont introduit un objet algébrique, le module de dérivation, censé contenir l'essentiel de l'information topologique sur le complémentaire d'une courbe algébrique dans $\C^2$. On renvoie à \cite{saito} pour plus de détails. Le module de dérivation est défini par:

\begin{dfn} 
Soit $\mathscr{C}$ une courbe algébrique définie par un polynôme $g$ $\in$ $\K[x,y]$. On note $Der(\mathscr{C})$ l'ensemble des champs de vecteurs logarithmiques, i.e. les champs de vecteurs tels que $g$ est un polynôme de Darboux:
\begin{align*}
Der(\mathscr{C})=\lbrace X \in Der(\C) \ \vert \ \text{il existe un cofacteur K  tel que} \ X.g=K.g \rbrace .
\end{align*}
\end{dfn}

Autrement dit, le module de dérivations est constitué des champs de vecteurs qui laisse invariante la courbe $g=0$. Le module de dérivations $Der(\mathscr{C})$ est une algèbre de Lie (voir \cite{saito}). Un corollaire direct du Théorème \ref{discinv} est:

\begin{theorem}
L'algèbre de Lie des symétries d'un tissu est une sous-algèbre de Lie de $Der(\Delta)$. 
\end{theorem}

En effet, si $X$ est un générateur infinitésimal d'une symétrie du tissu, alors $X.\Delta =\lambda \Delta$ par le théorème \ref{discinv} et $X \in Der(\Delta)$. La stabilité par crochet de Lie induit le théorème.

\section{Perspectives}
\label{perspective}

Les propriétés des tissus implicites peuvent en partie se lire dans leurs groupes de symétries. Ces groupes sont obtenus comme  groupes de symétries de l'équations différentielles polynomiale à coefficients analytiques représentant le tissu. Un tissu est une {\bf distribution de champs de vecteurs} au sens d'Olver (voir \cite{olver1}, \cite{olver2}). Il serait sans doute intéressant d'étudier les G-structures associées à ces distributions via la méthode d'équivalence de Cartan. 

Le problème de la régularité des algèbre de Lie des symétries entre dans l'étude de la régularité des solutions d'une équation aux dérivées partielles à coefficients analytiques. Il serait bon de voir les {\bf théorèmes de type Maillet} (voir \cite{gerard1,gerard2}, \cite{ger2} et Section \ref{regularite}) existant afin de préciser les intuitions d'Alain Hénaut sur ces objets. \\

{\bf Remerciements}: Les auteurs remercient Alain Henaut pour les avoir introduit aux tissus, les nombreuses discussions et remarques.

\end{document}